\documentclass[12pt]{amsart}
\usepackage[margin=1.25in]{geometry}
\usepackage{latexsym}
\usepackage{amsmath}
\usepackage{amssymb}
\usepackage{amsthm}
\usepackage{stmaryrd}
\usepackage{amscd}
\usepackage{enumerate, enumitem} 
\usepackage{amssymb} 
\usepackage{mathrsfs}
\usepackage[all]{xy}
\usepackage{color}
\usepackage{graphicx}
\usepackage{mathtools}
\usepackage{comment}
\usepackage{multirow}
\usepackage[pagebackref]{hyperref}
\hypersetup{colorlinks=true}
\usepackage{aliascnt}

\makeatletter
  
  \@addtoreset{equation}{thm}
  \makeatother

\title{Threefolds of globally $F$-regular type with nef anti-canonical divisor}

\author{Paolo Cascini}
\address{Department of Mathematics, Imperial College London, 180 Queen’s Gate, London SW7 2AZ, UK}
\email{p.cascini@imperial.ac.uk}

\author{Tatsuro Kawakami}
\address{Department of Mathematics, Graduate School of Science, Kyoto University, Kyoto 606-8502, Japan}
\email{tatsurokawakami0@gmail.com}

\author{Shunsuke Takagi}
\address{Graduate School of Mathematical Sciences, University of Tokyo, 3-8-1 Komaba,
Meguro-ku, Tokyo 153-8914, Japan}
\email{stakagi@ms.u-tokyo.ac.jp}

\def\ge{\geqslant}
\def\le{\leqslant}
\def\phi{\varphi}
\def\epsilon{\varepsilon}

\def\Spec{\operatorname{Spec}}

\def\Pic{\operatorname{Pic}}

\def\max{\operatorname{max}}
\def\Bs{\operatorname{Bs}}

\newcommand{\Q}{\mathbb{Q}}

\newcommand{\Z}{\mathbb{Z}}
\newcommand{\PP}{\mathbb{P}}

\newcommand{\sO}{\mathcal{O}}

\theoremstyle{plain}
\newtheorem{thm}{Theorem}[section] 
\newtheorem{cor}[thm]{Corollary}
\newtheorem{prop}[thm]{Proposition}
\newtheorem{conj}[thm]{Conjecture}

\newtheorem{lem}[thm]{Lemma}
\theoremstyle{definition} 
\newtheorem{defn}[thm]{Definition}

\newtheorem{eg}[thm]{Example} 

\theoremstyle{remark}
\newtheorem{rem}[thm]{Remark}
\newtheorem{ques}[thm]{Question}

\newtheorem{defn and notation}[thm]{Definition and Notation}

\newtheorem{cln}{Claim}

\theoremstyle{plain}
\newtheorem{theo}{Theorem}

\newtheorem{PROP}[theo]{Proposition}

\dedicatory{Dedicated to Professor Kei-ichi Watanabe on the occasion of his eightieth birthday.}

\begin{document}
\tolerance = 9999

\begin{abstract}
As a special case of a conjecture by Schwede and Smith \cite{SS10}, we prove that a smooth complex projective threefold with nef anti-canonical divisor is weak Fano if it is of globally $F$-regular type. 
\end{abstract}

\maketitle
\markboth{P.~CASCINI, T.~KAWAKAMI and S.~TAKAGI}{THREEFOLDS OF GLOBALLY $F$-REGULAR TYPE}


\section{Introduction}

Strong $F$-regularity, introduced by Hochster and Huneke \cite{HH89}, is a class of singularities in positive characteristic defined in terms of the Frobenius morphism. 
These singularities play a central role in the theory of tight closure, a powerful tool in commutative algebra in positive characteristic. 
During the 1990s, it turned out that they are also related to singularities in the minimal model program. 
It follows from the results of Smith \cite{Smith97}, Hara \cite{Hara98}, and Mehta-Srinivas \cite{Mehta-Srinivas97} (see also \cite{Hara-Watanabe}) that a normal $\Q$-Gorenstein singularity $(X,x)$ over a field of characteristic zero is klt if and only if its reduction modulo $p$ is strongly $F$-regular for almost all $p$. 

As a global analog of strongly $F$-regular singularities, Smith \cite{Smith00} introduced the notion of globally $F$-regular varieties, which, in particular,  only admit strongly $F$-regular singularities. 
As the name suggests, global $F$-regularity is a global condition on projective varieties in positive characteristic. 
For instance, their anti-canonical divisors are always big, and a Kawamata-Viehweg type vanishing theorem holds on these varieties. 
Globally $F$-regular varieties form an important subclass of Frobenius split varieties,  
including projective toric and Schubert varieties in positive characteristic. 
As a global version of the above correspondence between strongly $F$-regular and klt singularities, Schwede and Smith \cite{SS10} proposed the following conjecture: 
 
\begin{conj}[\textup{\cite[Question 7.1]{SS10}}]\label{Introconj}
A projective variety $X$ over a field of characteristic zero is of globally $F$-regular type if and only if it is of Fano type, that is, there exists an effective $\Q$-Weil divisor $\Delta$ on $X$ such that $-(K_X+\Delta)$ is ample $\Q$-Cartier and $(X,\Delta)$ is klt. 
\end{conj}

Here, a projective variety $X$ over a field of characteristic zero is said to be of \textit{globally $F$-regular type} if its reduction $X_p$ modulo $p$ is globally $F$-regular for almost all $p$.
Note that $\Q$-Gorenstein varieties of globally $F$-regular type have only klt singularities due to the correspondence between strongly $F$-regular and klt singularities. 

The ``if'' part of Conjecture \ref{Introconj} was proved by Schwede and Smith in loc.~cit., relying on \cite{Takagi}, while the “only if” part remains a wide open problem. 
The case of curves is straightforward, as globally $F$-regular curves are nothing but projective lines. 
Conjecture \ref{Introconj} has been proved in the case where $X$ is a Mori dream space \cite{GOST} or when $\dim X=2$ (\cite{GT}, \cite{HP15}, and \cite{Okawa17}).
In higher dimensions, unless the variety is a Mori dream space, almost nothing is known about this conjecture.  
When $X$ is a Mori dream space, the authors of \cite{GOST} run the anti-canonical minimal model program to reduce to the case where the anti-canonical divisor $-K_X$ is nef. 
In dimension two, the authors of \cite{GT} consider the Zariski decomposition $-K_X=P+N$ of the anti-canonical divisor $-K_X$ and focus on its positive part $P$, which is nef by definition. 
Therefore, in this paper, as a special case of Conjecture \ref{Introconj}, we discuss the case where $-K_X$ is nef $\Q$-Cartier. 

A klt projective variety is of Fano type if it is weak Fano, meaning that the anti-canonical divisor is nef and big. 
Since we assume that $-K_X$ is nef, it suffices to show that it is big. 
We know that its reduction $-K_{X_p}$ modulo $p$ is big for almost all $p$, because $X_p$ is globally $F$-regular by assumption. 
However, we do not know whether this condition implies the bigness of $-K_X$. 
There exists an example of non-big Cartier divisors whose reductions modulo $p$ are big for infinitely many $p$ (see Section \ref{big example}). 
Thus, we employ the strategy of utilizing the nefness of $-K_X$ more critically. 
The main difficulty then lies in the fact that nefness is not preserved under reduction modulo $p$ in general, while ampleness and semi-ampleness are preserved. 
In fact, there exists an example of nef Cartier divisors whose reductions modulo $p$ are not nef for almost all $p$ (see, for example, \cite{Langer15}).
Therefore, it remains unclear whether the higher cohomology of nef Cartier divisors on varieties of globally $F$-regular type vanishes, although that on globally $F$-regular varieties in positive characteristic does vanish (see Section \ref{vanishing section}).
Nefness and bigness are not well-behaved under reduction modulo $p$, which makes Conjecture \ref{Introconj} difficult to solve by standard birational geometric arguments. 

Nevertheless, we can show that if $D$ is a nef $\Q$-Cartier $\Q$-Weil divisor on a projective variety over a field of characteristic zero and if the stable base locus of $D$ has dimension at most one, then its reduction $D_p$ modulo $p$ is nef for almost all $p$ (Lemma \ref{lem:nef reduction}). 
This result has several important corollaries for varieties of globally $F$-regular type, including a cohomology vanishing result and a base point free result. 

\begin{PROP}[Corollaries \ref{cor:big}, \ref{cor:nef vanishing}, \ref{cor:abundance}]\label{Introprop}
Let $X$ be a projective variety of globally $F$-regular type over a field of characteristic zero, and let $D$ be a nef Cartier divisor on $X$ whose stable base locus $\mathrm{SB}(D)$ has dimension at most one. 
\begin{enumerate}[label=$(\arabic*)$]
\item
If $-K_X$ is nef $\Q$-Cartier and the stable base locus $\mathrm{SB}(-K_X)$ has dimension at most one, then $X$ is weak Fano. 
Thus, Conjecture \ref{Introconj} holds in this case. 
\item 
$H^i(X, \sO_X(D))=0$ for all integers $i>0$. 
\item 
If $\dim X \le 3$,  then $D$ is semi-ample. 
\end{enumerate}
\end{PROP}

In this paper, we give an affirmative answer to Conjecture \ref{Introconj} when $X$ is a smooth threefold with nef anti-canonical divisor by applying Proposition \ref{Introprop} to movable divisors on $X$ and nef effective divisors on some surfaces contained in $X$:

\begin{theo}[Theorem \ref{mainthm}]\label{Introthm}
Let $X$ be a smooth projective threefold over an algebraically closed field of characteristic zero.
If $X$ is of globally $F$-regular type and $-K_X$ is nef, then $X$ is weak Fano.
\end{theo}

We note that smooth rationally connected threefolds with nef anti-canonical divisor have been partially classified by Bauer-Peternell \cite{Bauer-Peternell} and Z.~Xie \cite{Xie20}, \cite{Xie23}. 
However, our proof does not directly rely on their classification. 
In fact, almost all the arguments in this paper remain valid without assuming the smoothness of $X$ (see Remark \ref{smoothness rem}).

\section*{Acknowledgments}
The authors are grateful to Jungkai Chen, Yoshinori Gongyo,  Kenta Sato and Calum Spicer for fruitful discussions. 
The third author also would like to express his gratitude to Imperial College London, where a part of this work was done, for their hospitality during the summer of 2018. 
The first author was partially supported by an EPSRC grant. 
The second author was supported by JSPS KAKENHI Grant numbers JP22KJ1771 and JP24K16897.
The third author was supported by JSPS KAKENHI Grant Numbers JP15KK0152, JP22H01112, JP23K22383 and JP20H00111.

\section{Preliminaries}

\subsection{Notation and terminology}

In this subsection, we summarise the notation and terminology that will be used in this paper. 
\begin{enumerate}
\item 
All schemes are assumed to be Noetherian and separated. 
A \textit{variety} is an integral scheme of finite type over a field, and a \textit{curve} (resp.~\textit{surface}, \textit{threefold}) is a variety of dimension one (resp.~two, three). 
\item 
An $\mathbb{F}_p$-scheme $X$ (resp.~a Noetherian ring $R$ of characteristic $p>0$) is said to be \textit{$F$-finite} if the absolute Frobenius morphism $F\colon  X \to X$ (resp. $F\colon \Spec R \to \Spec R$) is a finite morphism. 
A field $k$ of characteristic $p>0$ is $F$-finite if and only if the degree of the field extension $k/k^p$ is finite. 
If $X$ is a variety over an $F$-finite field, then $X$ is $F$-finite. 
\item Given a coherent sheaf $\mathcal{F}$ on a projective variety $X$ over a field $k$, $\dim_{k}\,H^i(X,\mathcal{F})$ is denoted by $h^i(X,\mathcal{F})$. 
\item A $\Q$-Weil divisor $D$ on a normal variety $X$ is called $\textit{big}$ if there exist an ample $\Q$-Cartier $\Q$-Weil divisor $H$ and an effective $\Q$-Weil divisor $E$ on $X$ such that $D \sim_{\Q} H+E$. 
\end{enumerate}

\subsection{Reduction modulo $p$}
In this subsection, we briefly explain how to reduce varieties, closed subschemes, divisors from characteristic zero to positive characteristic. 
Our main references are \cite{Hochster-Huneke20} and \cite{Mustata-Srinivas11}. 

Let $X$ be a variety over a field $k$ of characteristic zero and $Z$ be a closed subscheme of $X$.
By choosing an appropriate finitely generated $\Z$-subalgebra $A$ of $k$, we can construct a scheme $X_A$ of finite type over $A$ and a closed subscheme $Z_A$ such that the following commutative diagram with horizontal isomorphisms holds: 
\[\xymatrix{
X \ar[r]^{\cong \hspace*{3.5em}} &  X_A \times_{\Spec  A} \Spec  k\\
Z \ar[r]^{\cong \hspace*{3.5em}} \ar@{^{(}->}[u] & Z_{A} \times_{\Spec  A} \Spec  k. \ar@{^{(}->}[u]\\
}\]
We can enlarge $A$ by localizing at a single nonzero element and replacing $X_A$ and $Z_{A}$ with the corresponding open subschemes. 
Thus, by generic freeness, we may assume that $X_A$ and $Z_{A}$ are flat over $\Spec A$.
We refer to $(X_A, Z_A)$ as a \textit{model} of $(X, Z)$ over $A$.  
Enlarging $A$ further if necessary, we may also assume that $X_A$ is integral and that $Z_A$ is irreducible (resp.~reduced, integral) if $Z$ is so. 
Given a closed point $\mu \in \Spec A$, the fiber of $X_A$ (resp.~$Z_{A}$) over $\mu$ is denoted by $X_{\mu}$ (resp.~$Z_{\mu}$). 
Then $X_{\mu}$ is a scheme of finite type over the residue field $k(\mu)$ of $\mu$, which is a finite field since $A$ is a finitely generated algebra over $\mathbb{Z}$, and $Z_{\mu}$ is a closed subscheme of $X_{\mu}$.  
For general closed points $\mu \in \Spec A$, one can observe that $X_{\mu}$ is a variety and $Z_{\mu}$ is irreducible (resp.~reduced, integral) if $Z$ is so. 

Suppose that $X$ is normal and $D=\sum_i d_i D_i$ is a $\Q$-Weil divisor on $X$. 
We take a model $(X_A, D_{i,A})$ of $(X, D_i)$ over $A$. 
Enlarging $A$ if necessary, we may assume that $X_A$ is normal and $D_{i, A}$ is a prime divisor on $X_A$. 
We refer to $(X_A, D_A:=\sum_i d_i D_{i,A})$ as a \textit{model} of $(X, D)$ over $A$.   
For general closed points $\mu \in \Spec A$, it follows that $X_{\mu}$ is a normal variety over $\kappa(\mu)$ and $D_{i, \mu}$ is a prime divisor on $X_{\mu}$, so that $D_{\mu}:=\sum_i d_i D_{i, \mu}$ is a $\mathbb{Q}$-Weil divisor on $X_{\mu}$.  
If $D$ is Cartier (resp. $\Q$-Cartier), then so is $D_{\mu}$ for general closed points $\mu \in \Spec A$. 

The following lemma is useful to compare the cohomology of $X$ with that of $X_{\mu}$. 
\begin{lem}[\textup{\cite[Lemma 4.1]{Hara98}}]\label{cohomology comparison}
Let $D$ be a Weil divisor on a normal variety $X$ over a field $k$ of characteristic zero. 
Given a model $(X_A, D_A)$ of $(X, D)$ over a finitely generated $\Z$-subalgebra $A$ of $k$, 
one has isomorphisms 
\begin{align*}
H^i(X_A, \sO_{X_A}(D_A)) \otimes_A k & \cong H^i(X, \sO_X(D)), \\
H^i(X_A, \sO_{X_A}(D_A)) \otimes_A \kappa(\mu) & \cong H^i(X_{\mu}, \sO_{X_{\mu}}(D_{\mu})) 
\end{align*}
for general closed points $\mu \in \Spec A$ and for all $i \ge 0$. 
\end{lem}

Intersection numbers are preserved under reduction modulo $p$. 
We frequently use this well-known fact in this paper. 
\begin{lem}[\textup{cf.~\cite[Example 20.3.3]{FultonBook}}]\label{intersection}
Let $X$ be a projective variety over a field $k$ of characteristic zero, $Z \subset X$ be a closed subscheme of dimension $d$ and $D_1, \dots, D_d$ be $\Q$-Cartier $\Q$-Weil divisors on $X$.  
Given a model $(X_A, D_A)$ of $(X, D)$ over a finitely generated $\Z$-subalgebra $A$ of $k$, one has 
\[
D_1 \cdots D_d \cdot Z=D_{1,\mu} \cdots D_{d, \mu} \cdot Z_{\mu}
\]
for general closed points $\mu \in \Spec A$. 
\end{lem}
\begin{proof}
We may assume that all $D_i$ are Cartier. 
By an argument similar to that in the proof of \cite[Proposition 5.3]{Mustata-Srinivas11}, the problem can be reduced to the case where $A$ is a Dedekind domain. 
The assertion then follows from \cite[Example 20.3.3]{FultonBook}. 
\end{proof}

\subsection{Varieties of globally $F$-regular type}
In this subsection, we recall the definition and basic properties of varieties of globally $F$-regular type. 

\begin{defn}[\cite{Smith00}]
Let $X$ be a projective variety over an $F$-finite field of characteristic $p>0$. We say that $X$ is \textit{globally $F$-regular} if for every effective Cartier divisor $D$ on $X$, there exists an integer $e \ge 1$ such that the composition map 
\[
\sO_X \to F^e_*\sO_X \hookrightarrow F^e_*\sO_X(D)
\]
splits as an $\sO_X$-module homomorphism, where $\sO_X \to F^e_*\sO_X$ is the $e$-times iterated Frobenius morphism and $F^e_*\sO_X \hookrightarrow F^e_*\sO_X(D)$ is the $e$-times iterated Frobenius pushforward of the natural inclusion $\sO_X \hookrightarrow \sO_X(D)$. 
\end{defn}

\begin{prop}[\textup{\cite[Theorem 4.3]{SS10}}]\label{rem:GFR}
Let $X$ be a projective variety over an $F$-finite field of characteristic $p>0$. 
If $X$ is globally $F$-regular, then $X$ is of Fano type, that is, there exists an effective $\mathbb{Q}$-Weil divisor $\Delta$ on $X$ such that $-(K_X+\Delta)$ is ample $\Q$-Cartier and $(X,\Delta)$ is klt. 
In particular, $-K_X$ is big. 
\end{prop}

\begin{rem}
The construction of $\Delta$ in Proposition \ref{rem:GFR} heavily depends on the characteristic $p$. 
Therefore, Proposition \ref{rem:GFR} does not settle Conjecture \ref{Introconj}. 
\end{rem}

\begin{defn}
    Let $X$ be a projective variety over a field $k$ of characteristic zero.
    We say that $X$ is \textit{of globally $F$-regular type} if for a model $X_A$ of $X$ over a finitely generated $\mathbb{Z}$-subalgebra $A$ of $k$, there exists a dense open subset $U \subset \Spec A$ such that $X_{\mu}$ is globally $F$-regular for all closed points $\mu\in U$.
\end{defn}

We list basic properties of varieties of globally $F$-regular type that we  need later. 

\begin{prop}\label{thm:GFRtype to RC}\label{rem:GFRtype}
Suppose that $X$ is a projective variety over a field of characteristic zero. 
\begin{enumerate}[label=$(\arabic*)$]
\item $($\cite[Theorem 5.1]{SS10}$)$ If $X$ is of Fano type, then $X$ is of globally $F$-regular type.  
\item $($\cite[Theorem 3.9]{Hara-Watanabe}$)$ If $X$ is $\Q$-Gorenstein of globally $F$-regular type, then $X$ has only klt singularities. 
\item $($\cite[Theorem 5.8]{GLP$^+$15}$)$ If $k$ is algebraically closed and $X$ is $\Q$-Gorenstein of globally $F$-regular type, then $X$ is rationally connected. 
\end{enumerate}
\end{prop}

Being of globally $F$-regular type is preserved under crepant birational morphisms and small birational maps.  

\begin{lem}[\textup{\cite[Corollary 6.4]{SS10}, \cite[Lemma 2.7]{GT}, \cite[Lemma 2.14]{GOST}}]\label{lem:crepant and GFR}\label{lem:small and GFR}
Let $\phi\colon X\dasharrow X'$ be a birational map between normal projective varieties over a field $k$ of characteristic zero. 
\begin{enumerate}[label=$(\arabic*)$]
\item 
If $\phi$ is a morphism and $X$ is of globally $F$-regular type, then $X'$ is also of globally $F$-regular type. 
\item 
Suppose that one of the following conditions holds: 
\begin{enumerate}[label=\textup{(\roman*)}]
\item 
$\phi$ is a morphism, $X'$ is $\Q$-Gorenstein, and $K_X=\phi^{*}K_{X'}$. 
\item 
$\phi$ is small. 
\end{enumerate}
Then $X$ is of globally $F$-regular type if and only if so is $X'$.
\end{enumerate}
\end{lem}

\subsection{Riemann--Roch theorem}
We will use the following version of Riemann-Roch theorem for singular varieties in the proof of Theorem \ref{Introthm}. 
\begin{thm}
    Let $X$ be a Gorenstein terminal projective threefold over an algebraically closed field of characteristic zero and let $D$ be a Cartier divisor on $X$.
    Then the following hold:
    \begin{align*}
        \chi(X,\sO_X(D))&=\frac{1}{12}D\cdot (D-K_X)\cdot (2D-K_X)+\frac{1}{12}D \cdot c_2(X)+\chi(X,\sO_X)\,\,\text{and}\,\, \\
        \chi(X,\sO_X)&=\frac{1}{24}(-K_X) \cdot c_2(X).
    \end{align*}
    In particular, we have
    \begin{equation}\label{RR for -K}
        \chi(X,\sO_X(-K_X))=\frac{1}{2}(-K_X)^3+3\chi(X,\sO_X).
    \end{equation}
\end{thm}
\begin{proof}
    Taking $r=1$ and $m=0$ in \cite[Corollary 10.3]{Reid(Young)} yields the assertion. 
\end{proof}

\section{A vanishing theorem for varieties of globally ${F}$-regular type}\label{vanishing section}

In this section, we study a vanishing theorem for varieties of globally $F$-regular type. 

Smith \cite{Smith00} proved that the higher cohomology of a nef divisor on a globally $F$-regular variety vanishes. 

\begin{prop}[\textup{cf.~\cite[Theorem 4.2 (2), Corollary 4.3]{Smith00}}]\label{thm:nef vanishing for GFR}
    Let $X$ be a globally $F$-regular projective variety over an $F$-finite field of characteristic $p>0$ and $D$ be a $\Q$-Cartier Weil divisor on $X$. 
\begin{enumerate}[label=$(\arabic*)$]
    \item If there exists an effective Weil divisor $B$ on $X$ and an integer $n_0 \ge 1$ such that $H^i(X, \sO_X(nD+B))=0$ for all $n \ge n_0$, then $H^i(X, \sO_X(D))=0$. 
    \item If $D$ is nef, then $H^i(X,\sO_X(D))=0$ for all $i>0$.
\end{enumerate}
\end{prop}

As an immediate corollary of Proposition \ref{thm:nef vanishing for GFR}, we have the following vanishing theorem for varieties of globally $F$-regular type. 
\begin{cor}\label{kodaira vanishing}
Let $X$ be a projective variety over a field $k$ of characteristic zero, and suppose that $X$ is of globally $F$-regular type.   
If $D$ is a semi-ample $\Q$-Cartier Weil divisor on $X$, then $H^i(X, \sO_X(D))=0$ for all $i>0$. 
\end{cor}

\begin{proof}
Suppose that $(X_A, D_A)$ is a model of $(X,D)$ over a finitely generated $\Z$-subalgebra $A$ of $k$, and consider the flat deformation $X_A \to \Spec A$. 
Since $D \cong D_A \otimes_A k$ is semi-ample, the restriction of $D_A$ to the generic fiber $X_A \otimes_A \mathrm{Frac}\, A$ is semi-ample. 
Then the restriction of $D_A$ to a general closed fiber $X_{\mu}$ is also semi-ample. 
It follows from Proposition \ref{thm:nef vanishing for GFR} (2) that $H^i(X_{\mu}, \sO_{X_{\mu}}(D_{\mu}))=0$ for general closed points $\mu \in \Spec A$, which implies that $H^i(X, \sO_X(D))=0$ by Lemma \ref{cohomology comparison}. 
\end{proof}

As a generalization of Corollary \ref{kodaira vanishing}, we expect that the following conjecture holds. 

\begin{conj}\label{vanishing conj}
Let $X$ be as in Corollary \ref{kodaira vanishing}. 
If $D$ is a nef $\Q$-Cartier Weil divisor on $X$, then $H^i(X, \sO_X(D))=0$ for all $i>0$. 
\end{conj}

\begin{rem}
\begin{enumerate}
\item Conjecture \ref{vanishing conj} follows from Conjecture \ref{Introconj}. 

\item If $D$ is Cartier, then Conjecture \ref{vanishing conj} is presented as a proposition in \cite{Smith00} without proof, but to the best of the authors' knowledge, it remains open. 
Since nefness is not preserved under reduction modulo $p$ (see, for example, \cite[Section 8]{Langer15}, which gives an example of a nef Cartier divisor $D$ whose reduction modulo $p$ is not nef for almost all $p$), it does not follow from Proposition \ref{thm:nef vanishing for GFR}. 

\item  One might consider Conjecture \ref{vanishing conj} as a corollary of a result of Schoutens \cite[Corollary 6.6]{Schoutens05}, but there is a serious gap in the proof of \cite[Theorem 2.15]{Schoutens05}, which is crucial for \cite[Corollary 6.6]{Schoutens05}. 
In Example \ref{counterexample}, we will give a counterexample to \cite[Corollary 2.16]{Schoutens05}, which is an immediate corollary of \cite[Theorem 2.15]{Schoutens05}. 
\end{enumerate}
\end{rem}

If the following question has an affirmative answer, then Conjecture \ref{vanishing conj} would follow from Proposition \ref{thm:nef vanishing for GFR} (1).  
\begin{ques}\label{vanishing ques}
Let $X$ be a projective variety over a field $k$ of characteristic zero, and let $\mathcal{E}$ and $\mathcal{L}$ be invertible sheaves on $X$. 
We fix a positive integer $i$, and suppose that there exists an integer $n_0 \ge 1$ such that $H^i(X, \mathcal{E}\otimes \mathcal{L}^n)=0$ for all $n \ge n_0$. 
Then for a model $(X_A, \mathcal{E}_A, \mathcal{L}_A)$ of $(X, \mathcal{E}, \mathcal{L})$ over a finitely generated $\Z$-subalgebra $A$ over $k$, does there exist a dense subset $S \subseteq \Spec A$ of closed points such that 
$H^i(X_{\mu}, \mathcal{E}_{\mu} \otimes \mathcal{L}^n_{\mu})=0$ for all $n \ge n_0$ and for all $\mu \in S$? 
\end{ques}

However, the above question has a negative answer in general. 

\begin{eg}\label{counterexample}
Let $E$ be an elliptic curve over $\Q$, $\mathcal{E}=\sO_X$ and $\mathcal{L}$ be a numerically trivial non-torsion line bundle on $E$. Then 
\[
H^1(E, \mathcal{E}\otimes \mathcal{L}^n)=H^1(E, \mathcal{L}^n)=0    
\]
for all $n \ge 1$. 
On the other hand, since every numerically trivial line bundle on a projective scheme defined over a finite field is torsion, $\mathcal{L}_p$ is a torsion line bundle on $E_p$ for all but finitely many primes $p$ (here we consider a model of $E$ over $\Z[1/N]$ for some integer $N \ge 1$). 
If $n$ is a multiple of the order of $\mathcal{L}_p$, then 
\[H^1(E_p, \mathcal{E}_p\otimes \mathcal{L}_p^n)=H^1(E_p, \sO_{E_p}) \ne 0.\]
This gives a counterexample to Question \ref{vanishing ques} and \cite[Corollary 2.16]{Schoutens05}. 
\end{eg}


\section{On bigness under reduction modulo $p$}\label{big example}
In this section, we construct an example of a nef but not big Cartier divisor over a complex projective surface, whose reduction modulo $p$ is big but not nef for infinitely many primes $p$. 

Our example is a variation of \cite[Remark 3, pag. 3]{yang}. We follow a similar notation to that in \cite{DP} and \cite{AG}. Let $k=\mathbb C$, let $L\subset k$ be a real quadratic field, and let $R$ be the ring of integers of $L$. We denote by $\Delta$ the discriminant of $R$ over $\mathbb Z$. Fix a positive integer $N\ge 4$, let $\mu_N\subset k$ be the group generated by the $N$-th roots of  unity and let $\mathcal M$ be the moduli space of abelian surfaces with $\mu_N$-level structure. By \cite[Theorem 2.2]{DP},  $\mathcal M$ is smooth over $\Spec A$, where $A\coloneqq \mathbb Z[1/N\Delta]$. 
The Hodge bundle $\mathbb E$ on $\mathcal M$ is a locally free sheaf of rank $2$ over $\mathcal M$ \cite[\S 4.1]{AG} and there exist invertible sheaves $\mathbb L_1$ and $\mathbb L_2$ on $\mathcal M$ such that $\mathbb  E=\mathbb L_1\oplus \mathbb L_2$ and, for any $a,b\in \mathbb Z$, the sections of $\mathbb L_1^{\otimes a}\otimes \mathbb L_2^{\otimes b}$ are modular forms of weight $(a,b)$.

We denote by $V$ the Satake compactification of $\mathcal M$.
Then there exist $\ell_1$ and
$\ell_2$ Cartier divisors on $V$ such that $\sO_{\mathcal M} (\ell_i)\cong \mathbb L_i$ for $i=1,2$
\cite[Proposition 7.13]{LanSuh13}.
Let $X\coloneqq V\times_{\Spec A}\Spec k$ and let $\ell_{i,0}$ be the restriction of $\ell_i$ on $X$ for $i=1,2$.

We fix a rational prime number  $p$. Let $V_p$ be the fiber of $V\to \Spec A$ over $p$ and let $\ell_{i,p}$ be  the restriction of $\ell_i$ over $V_p$, for $i=1,2$. 
Then, as in \cite[\S 8]{AG} we have that 
$\ell_{1,p}^2=\ell_{2,p}^2=0$.

If $p$ is split in $L$ then \cite[Threom 8.2.4]{AG} implies that $a\ell_{1,p}+b\ell_{2,p}$ is ample if and only if $a,b>0$. In particular $\ell_{1,p}$ and $\ell_{2,p}$ are nef divisors. 
On the other hand, if $p$ is inert in $L$, then \cite[Threom 8.1.1]{AG} implies that $a\ell_{1,p}+b\ell_{2,p}$ is ample if and only if $a,b>0$ and $a/p <b <pa$. Moreover, as in the proof of \cite[Threom 8.1.1]{AG}, it follows that $p\ell_{1,p}-\ell_{2,p}$ and $p\ell_{2,p}-\ell_{1,p}$ are linearly equivalent to  effective divisors. Thus, 
$$\ell_{1,p} = \frac 1 {p+1} ( (\ell_{1,p}+\ell_{2,p})+ (p\ell_{1,p}-\ell_{2,p}))$$ 
is a big divisor, as it is linearly equivalent to a sum of an ample divisor and an effective divisor. On the other hand, $\ell_{1,p}$ is not nef, since $\ell_{1,p}+t\ell_{2,p}$ is not ample for any sufficiently small number $t>0$.

In conclusion, we have that $\ell_{1,0}$ and $\ell_{2,0}$ are nef but not big divisors on $X$, whose restrictions modulo $p$ is big but not nef for infinitely many $p$. 


\section{The case of small stable base loci}

In this section, we prove Proposition \ref{Introprop} and, in particular, establish Conjecture \ref{Introconj} when $-K_X$ is nef and $\dim \mathrm{SB}(-K_X)\leq 1$.

First, we recall the definition of the stable base locus and the numerical dimension. 
\begin{defn}
Let $X$ be a $d$-dimensional normal projective variety over a field and $D$ be a $\Q$-Cartier $\Q$-Weil divisor on $X$.
\begin{enumerate}
\item 
We fix a positive integer $r>0$ such that $rD$ is Cartier.
The \textit{stable base locus} $\mathrm{SB}(D)$ of $D$ is defined as
\[
\mathrm{SB}(D)= \bigcap_{m\geq 1}\mathrm{Bs}(mrD)_{\mathrm{red}},
\]
where $\mathrm{Bs}(D')_{\mathrm{red}}$ denotes the base locus of a Cartier divisor $D'$ with reduced structure.
Note that the definition of $\mathrm{SB}(D)$ does not depend on the choice of $r$.
\item 
Suppose that $D$ is nef. Then the \textit{numerical dimension} $\nu(D)$ of $D$ is defined as 
\[
\nu(D)=\max\{i \in \Z_{\ge 0} \; | \; D^{i} \cdot H^{d-i} \ne 0\},
\]
where $H$ is an ample Cartier divisor on $X$. 
\end{enumerate}
\end{defn}

The following lemma is a key to this section. 

\begin{lem}\label{lem:nef reduction}
Let $X$ be a projective variety over a field $k$ of characteristic zero and $D$ be a nef $\Q$-Cartier $\Q$-Weil divisor on $X$.
Suppose that the stable base locus $\mathrm{SB}(D)$ of $D$ has dimension at most one. 
Given a model $(X_A, D_A)$ of $(X, D)$ over a finitely generated $\Z$-subalgebra $A$ of $k$, 
then $D_{\mu}$ is nef and $\nu(D_{\mu})=\nu(D)$ for a general closed point $\mu \in \Spec A$. 
\end{lem}
\begin{proof}
Take a sufficiently divisible $m$ such that $\mathrm{SB}(D)=\mathrm{Bs}(mD)_{\mathrm{red}}$. 
Let $\sum_{i}a_iC_i$ denote the dimension one part of the base locus $\Bs(mD)$.
By Lemma \ref{cohomology comparison}, 
\[H^0(X, \sO_X(mD))_{\mu} = H^0(X_{\mu}, \sO_{X_{\mu}}(mD_{\mu})),\]
so that $\Bs(mD_{\mu})= \Bs(mD)_{\mu}$ for a general closed point $\mu \in \Spec A$. 

First, we  show that $D_{\mu}$ is nef, that is, $D_{\mu}\cdot F^{\mu}\geq 0$ for every curve $F^{\mu}$ on $X_{\mu}$. 
We may assume that $F^{\mu}$ is contained in $\Bs(mD_{\mu})$, and thus $F^{\mu}=C_{i,\mu}$ for some $i$.
Then, by Lemma \ref{intersection} we have 
\[
D_{\mu}\cdot F^{\mu}=D_{\mu}\cdot C_{i,\mu}= D\cdot C_{i}\geq 0.
\] 

Next, we  show that $\nu(D)=\nu(D_{\mu})$ for a general closed point $\mu \in \Spec A$. 
For an ample Cartier divisor $H$ on $X$, we have  that $H_{\mu}$ is an ample Cartier divisor on $X_{\mu}$, and by Lemma \ref{intersection} we have 
\[
D^{j}\cdot H^{\dim X-j}=D_{\mu}^{j}\cdot H_{\mu}^{\dim X-j} 
\] 
for a general closed point $\mu \in \Spec A$, which implies the assertion. 
\end{proof}

\begin{cor}\label{cor:big}
Let $X$ be a projective variety over a field $k$ of characteristic zero such that $-K_X$ is nef $\Q$-Cartier and $\dim \mathrm{SB}(-K_X)\leq 1$. 
If $X$ is of globally $F$-regular type, 
then $-K_X$ is big. In particular, Conjecture \ref{Introconj} holds in this case.
\end{cor}
\begin{proof}
Given a model of $X$ over a finitely generated $\Z$-subalgebra $A$ of $k$, by Lemma \ref{lem:nef reduction}, $-K_{X_{\mu}}$ is nef for a general closed point $\mu \in \Spec A$. 
On the other hand, since $X_{\mu}$ is globally $F$-regular, it follows from Proposition \ref{rem:GFR} that $-K_{X_{\mu}}$ is big for a general closed point $\mu \in \Spec A$. 
Therefore, by Lemma \ref{intersection} we have 
\[(-K_{X})^3=(-K_{X_{\mu}})^3>0,\] 
which implies that $-K_X$ is big.
\end{proof}

\begin{cor}\label{cor:nef vanishing}
Let $X$ be a projective variety over a field $k$ of characteristic zero, and suppose that $X$ is of globally $F$-regular type.
If $D$ is a nef $\Q$-Cartier Weil divisor on $X$ such that $\dim \mathrm{SB}(D)\leq 1$, then Conjecture \ref{vanishing conj} holds for $D$, that is, 
\[
H^i(X,\sO_X(D))=0
\]
for all $i>0$. 
\end{cor}
\begin{proof}
Given a model of $(X, D)$ over a finitely generated $\Z$-subalgebra $A$ of $k$, by Lemma \ref{lem:nef reduction}, $D_{\mu}$ is nef for a general closed point $\mu \in \Spec A$. 
It then follows from Proposition \ref{thm:nef vanishing for GFR} (2) that $H^i(X_{\mu},\sO_{X_{\mu}}(D_{\mu}))=0$ for all $i>0$ and for a general closed point $\mu \in \Spec A$. 
Finally, Lemma \ref{cohomology comparison} establishes the assertion. 
\end{proof}

As another consequence of Lemma \ref{lem:nef reduction}, we prove a special case of the base point free theorem for threefolds of globally $F$-regular type. 

\begin{cor}\label{cor:abundance}
    Let $X$ be a projective threefold over a field $k$ of characteristic zero, and suppose that $X$ is of globally $F$-regular type.
    If $D$ is a nef $\Q$-Cartier $\Q$-Weil divisor such that $\dim\,\mathrm{SB}(D)\leq 1$, then $D$ is semi-ample.
\end{cor}
\begin{proof}
Suppose that $(X_A, D_A)$ is a model of $(X, D)$ over a finitely generated $\Z$-subalgebra $A$ of $k$. 
We fix a general closed point $\mu \in \Spec A$ such that $k(\mu)$ is of characteristic greater than five.
Let ($X_{\overline{\mu}}, D_{\overline{\mu}})$ denote the base change of $(X_{\mu}, D_{\mu})$ to an algebraic closure $\overline{\kappa(\mu)}$ of $\kappa(\mu)$. 
Note that by Lemma \ref{lem:nef reduction}, $D_{\mu}$, and therefore $D_{\overline{\mu}}$, is nef. 
Since $X_{\overline{\mu}}$ is globally $F$-regular by \cite[Proposition 2.6]{GLP$^+$15} and, in particular, of Fano type by Proposition \ref{rem:GFR}, 
it follows from the base point free theorem \cite[Theorem 1.2]{Birker-Waldron} that  $D_{\overline{\mu}}$, and therefore $D_{\mu}$, is semi-ample. 
Take a positive integer $m \ge 1$ such that $mD_{\mu}$ is a free Cartier divisor.
Then $H^i(X_{\mu},\sO_{X_{\mu}}(mD_{\mu}))=0$ for all $i>0$ by Proposition \ref{thm:nef vanishing for GFR} (2).

Let $R\coloneqq \widehat{A_{\mu}}$ denote the $\mu$-adic completion of the localization of $A$ at $\mu$. 
By taking the base change by $\Spec\,R\to \Spec\,A$, the following hold:
\begin{enumerate}
\item $X_{\mu}$ and $\sO_{X_{\mu}}(mD_{\mu})$ have liftings $X_R$ and $\sO_{X_R}(mD_R)$ over $R$; 
\item the base change of the generic fiber of $X_R\to \Spec R$ to $k$ is isomorphic to $X$; and  
\item $\sO_{X_R}(mD_R)\otimes_R k \cong \sO_X(mD)$.
\end{enumerate}
By the upper semicontinuity of the dimension of the cohomology group of fibers \cite[Chapter III, Theorem 12.8]{Har}, we obtain $H^i(X_R, \sO_{X_R}(mD_R))=0$ for all $i>0$.
Then Grauert's theorem implies that the natural map
\[
    H^0(X_R, \sO_{X_R}(mD_R))\otimes_{R} k(\mu)\to H^0(X_{\mu}, \sO_{X_{\mu}}(mD_{\mu}))
\]
is an isomorphism.

Now we  show that $mD$ is free.  
Since $\Bs(mD_R)$ is closed in $X_R$ and $X_R$ is projective over $R$, the subscheme $\Bs(mD_R)$ would dominate $\Spec\,R$ if it were not empty. 
However, by the above isomorphism, together with the fact that $mD_{\mu}$ is free, the base locus $\Bs(mD_R)$ must be empty. 
Then, by flat base change, 
\[H^0(X_R, \sO_{X_R}(mD_R))\otimes_{R} k \cong H^0(X,\sO_X(mD)),\]
leading to the conclusion that $mD$ is free.
\end{proof}

\section{Proof of the Main Theorem}
Throughout this section, let $k$ denotes an algebraically closed field of characteristic zero. 
We  prove the main theorem of this paper: 

\begin{thm}\label{mainthm}
Let $X$ be a smooth projective threefold over $k$ with $-K_X$ nef. 
If $X$ is of globally $F$-regular type, then $-K_X$ is big. 
In particular, Conjecture \ref{Introconj} holds in this case.
\end{thm}

The proof is divided into three cases based on the numerical dimension of $X$. 

\subsection{The case $\nu(-K_X)=0$}

\begin{prop}\label{prop:nu=0}
    Let $X$ be a normal projective variety over $k$ of globally $F$-regular type with $-K_X$ nef.
    Then $\nu(-K_X)\neq 0$.
\end{prop}
\begin{proof}
    Suppose by contradiction that $\nu(-K_X)= 0$.
    Since $H^1(X, \sO_X)=0$ by Corollary \ref{kodaira vanishing}, the numerical equivalence for $\Q$-Cartier $\Q$-Weil divisors on $X$ is the same as their $\Q$-linear equivalence, and therefore, $-K_X\sim_{\Q}0$.
    It then follows from Corollary \ref{cor:big} that $-K_X$ is big, which contradicts the fact that $\nu(-K_X)= 0$.
\end{proof}

\subsection{The case $\nu(-K_X)=2$}

In this subsection, we aim to prove the following proposition:

\begin{prop}\label{prop:nu=2}
    Let $X$ be a normal projective threefold over $k$ of globally $F$-regular type such that $-K_X$ is nef Cartier.
    Then $\nu(-K_X)\neq 2$.
\end{prop}

We start with two auxiliary lemmas.

\begin{lem}[\textup{cf.~\cite[Proposition 4.1]{Xie20}}]\label{lem:flip}
    Let $X$ be a Gorenstein terminal $\Q$-factorial projective threefold over $k$ such that $-K_X$ is nef and $h^0(X,\sO_X(-K_X))\geq 2$.
    Let $|-K_X|=|M|+B$ be the decomposition into the movable part $|M|$ and the fixed part $B$.
    Then there exists a sequence 
    \[
    \phi\colon X\dasharrow X_1 \dasharrow\cdots \dasharrow X_s \eqqcolon X' 
    \]
    of flops such that
    \begin{enumerate}[label=\textup{(\arabic*)}]
        \item $X'$ is a Gorenstein terminal $\Q$-factorial threefold,
        \item $h^0(X',\sO_{X'}(-K_{X'}))\geq 2$,
        \item $-K_{X'}$ is nef with $\nu(K_X)=\nu(K_{X'})$, 
        \item $M'\coloneqq \phi_{*}M$ is movable, and
        \item $M'\coloneqq \phi_{*}M$ is nef.
    \end{enumerate}
\end{lem}
\begin{proof}
The proof is quite similar to that of \cite[Proposition 4.1]{Xie20}, but we include it here for the sake of completeness. 

    We may assume that $M$ is not nef.
    Fix a member $H\in |M|$ and a positive rational number $0< \epsilon \ll 1$ such that $(X, \epsilon H)$ is terminal.
    Since $-K_X$ is nef and $H$ is not nef, by \cite[Lemma 2.5]{Xie20} there exists a $(K_X+\epsilon H)$-negative extremal contraction $\phi\colon X\to Y$ such that $\ell \cdot H<0$ for every contracted curve $\ell$.
    As $|M|$ is movable, we observe that $\phi$ is small birational.
    Flipping contractions do not occur on Gorenstein terminal $\Q$-factorial threefolds (see \cite[Th\'eor\`em 0]{Benveniste}), so we have $K_X\cdot \ell=0$.
    In particular, $\phi$ is a flopping contraction.
    
    Let $\phi^{+}\colon X^{+}\to Y$ be the flop of $\phi$ (which exists by \cite[Theorem 6.14]{KM98}), and set $\psi \coloneqq (\phi^{+})^{-1}\circ \phi \colon X\dasharrow X^{+}$.
    Then $X^{+}$ is $\Q$-factorial Gorenstein terminal as shown in \cite[Theorem 6.15]{KM98}.
    By taking into account that $K_X\cdot \ell=0$ and $\phi_{*}K_X=K_Y$, it follows from the cone theorem \cite[Theorem 3.7]{KM98} that $K_Y$ is Cartier and $K_X=\phi^{*}K_Y$. 
    By pushing forward via $\psi$, this implies $K_{X^{+}}=(\phi^{+})^{*}K_{Y}$. 
    Thus, $-K_{X^{+}}$ is nef with $\nu(-K_X)=\nu( -K_Y)=\nu( -K_{X^{+}})$,  and 
    \[
    h^0(X^{+},\sO_{X^{+}}(-K_{X^{+}}))=h^0(X,\sO_X(-K_X))\geq 2.
    \]
    Since $\psi$ is small, there is a one-to-one correspondence between Weil divisors on $X$ and those on $X'$, given by the pushforward via $\psi$.
    In particular, 
   $M^{+}\coloneqq \psi_{*}M$ is movable.
   Therefore, the conditions (1)--(4) are satisfied.

   Finally, by  termination of three-dimensional terminal flops \cite[Corollary 6.19]{KM98}, the condition (5) is satisfied after repeating this procedure finitely many times.
\end{proof}

\begin{lem}\label{lem:restriction}
    Let $f\colon X\to Y$ be a surjective morphism between normal projective varieties, and suppose that $Y$ is a smooth curve.
    Let $D$ be a big $\Q$-Cartier $\Q$-Weil divisor on $X$ and $F$ be an irreducible fiber of $f$. 
   Then $D|_F$ is big.
\end{lem}
\begin{proof}
    Take an ample $\Q$-Cartier $\Q$-Weil divisor $A$ and an effective $\Q$-Weil divisor $E$ on $X$ such that $D\sim_{\Q} A+E$.
    Since $Y$ is a smooth curve, all fibers of $f$ are linearly equivalent.  
    If $E$ contains $F$ in its support, then by replacing $F$ with another fiber $F'$ of $f$, we can take an effective $\Q$-Weil divisor $E'$ on $X$ that is $\Q$-linearly equivalent to $E$ and does not contain $F$ in its support. 
    Therefore, we may assume that $E$ does not contain $F$ in its support. 
    Then $D|_F\sim_{\Q} A|_F+E'|_F$ is big because $A|_F$ is ample and $E'|_F$ is effective.
\end{proof}

\begin{proof}[Proof of Proposition \ref{prop:nu=2}]
    Since $X$ is klt by Proposition \ref{rem:GFRtype} (2) and $K_X$ is Cartier, it follows that $X$ is canonical.
    We take a $\Q$-factorial terminalization $f\colon Y\to X$, that is, $f$ is a projective birational morphism such that $Y$ is $\Q$-factorial terminal and $K_Y=f^{*}K_X$ (such an $f$ exists by \cite[P.~360, Theorem (b)]{Reid(Young)}).
    Note that  $\nu(-K_Y)=\nu(f^{*}(-K_X))=\nu(-K_X)$, 
    and $Y$ is of globally $F$-regular type by Lemma \ref{lem:crepant and GFR}. Thus, replacing $X$ with $Y$, we may assume that $X$ is $\Q$-factorial Gorenstein terminal.

    Suppose by contradiction that $\nu(-K_X)=2$.    
    By a variant of Kawamata-Viehweg vanishing \cite[Lemma 2.1]{Lazic-Peternell17}, we have $h^2(X, \sO_Y(-K_X))=0$.
    Since $(-K_X)^3=0$ by assumption and $\chi(X,\sO_X)=1$ by Proposition \ref{kodaira vanishing}, 
    the Riemann-Roch theorem \eqref{RR for -K} implies that
    \[
    h^0(X, \sO_X(-K_X))\geq \chi(X, \sO_X(-K_X))=3.
    \]
    
    Let $|-K_X|=|M|+B$ be the decomposition into the movable part $|M|$ and the fixed part $B$.
    By Lemma \ref{lem:flip} and Lemma \ref{lem:small and GFR}, after repeatedly replacing $X$ with its flop, we can assume that $M$ is nef.
    Then, by Corollary \ref{cor:abundance}, $M$ is semi-ample. 
    We fix an integer $m \ge 1$ such that $|mM|$ is base point free.
    Let $\phi\colon X\to Z$ be the morphism associated to the complete linear system $|mM|$. 
    As in the proof of \cite[Lemma 4.2]{Xie20}, we can observe that 
    \[M^3=M^2\cdot K_X=M\cdot K_X^2=0,\] and in particular, $M$ is not big.
    Since $h^0(X,\sO_X(M))=h^0(X, \sO_X(-K_X))\geq 3$, it follows that $1\leq \dim Z\leq 2$.
    Let $F$ be a general fiber of $\phi$. 

    \textbf{Step 1($\dim Z=1$ case).}
    First, we consider the case $\dim Z=1$.
    By Bertini theorem and since $h^0(X, \sO_X(-K_X))\geq 3$, we have that $F$ is a smooth irreducible member of $|mM|$ and $-K_{F}=-K_{X}|_{F}$ is nef and linearly equivalent to an effective divisor.
    Given a model $(X_A, F_A)$ of $(X, F)$ over a finitely generated $\Z$-subalgebra $A$ of $k$, we fix a general closed point $\mu \in \Spec A$.  
    By Lemma \ref{lem:nef reduction}, 
    \[-K_{F_{\mu}}=-(K_{X_{\mu}}|_{F_{\mu}})= -(K_{X}|_{F})_{\mu}=(-K_{F})_{\mu}\] 
    is nef.
    Since  
    $-K_{X_{\mu}}$ is big by Proposition \ref{rem:GFR} and 
    $F_{\mu}$ is irreducible, 
    $-K_{F_{\mu}}=-K_{X_{\mu}}|_{F_{\mu}}$ is big by Lemma \ref{lem:restriction}.
    This leads to  
    \[
    0<(-K_{F_{\mu}})^2=(-K_F)^2=(-K_X)^2\cdot F=K_X^2 \cdot mM=0,
    \]
    which is a contradiction.

    \textbf{Step 2($\dim\,Z=2$ case).}
    Next, we discuss the case $\dim Z=2$. 
    Since $F$ is a smooth irreducible curve and $\deg K_F=K_X\cdot F=K_X\cdot (mM)^2=0$, $F$ is an elliptic curve. 
    Let $\eta$ be the generic point of $Z$ and $X_{\eta}$ denote the generic fiber of $\phi$. 
    Then $h^1(X_{\eta}, \sO_{X_{\eta}})=1$. 
    
    On the other hand, suppose that $(X_A, Z_A)$ is a model of $(X,Z)$ over a finitely generated $\Z$-subalgebra $A$ of $k$, and fix a general closed point $\mu \in \Spec A$. 
    Since $X_{\mu}$ is globally $F$-regular, 
    it follows from \cite[Theorem 2.1]{GLP$^+$15} that a general closed fiber $G$ of $\phi_{\mu}:X_{\mu} \to Z_{\mu}$ is globally $F$-regular and, in particular, $H^1(G, \sO_G)=0$ by Proposition \ref{thm:nef vanishing for GFR}. 
    Then 
    \[H^1((X_{\eta})_{\mu}, \sO_{(X_{\eta})_{\mu}})=H^1((X_{\mu})_{\eta_{\mu}},\sO_{(X_{\mu})_{\eta_{\mu}}})=0,\] 
    where $\eta_{\mu}$ is the generic point of $Z_{\mu}$. 
    This contradicts Lemma \ref{cohomology comparison}. 
\end{proof}

\subsection{The case $\nu(-K_X)=1$}

In this subsection, we aim to show the following proposition:

\begin{prop}\label{prop:nu=1}
    Let $X$ be a smooth projective threefold over $k$ of globally $F$-regular type with $-K_X$ nef.
    Then $\nu(-K_X)\neq 1$.
\end{prop}

First, we consider the case of Mori fiber spaces. 

\begin{lem}\label{lem:nu=1,non-birat}
    Let $X$ be a Gorenstein terminal $\Q$-factorial projective threefold over $k$ of globally $F$-regular type such that $-K_X$ is nef. 
    Let $f\colon X\to Y$ be a $K_X$-negative extremal  contraction such that
 $\dim Y<3$. Then $\nu(-K_X)\neq 1$.
\end{lem}
    \begin{proof}
    Suppose by contradiction that $\nu(-K_X)=1$. In particular, this implies that $\dim Y>0$.
     If $\dim Y=1$, then  \[(-K_X)^2\cdot F=(-K_F)^2>0\] for a general fiber $F$ of $f$, which contradicts the assumption that $\nu(-K_X)=1$.
     
    In what follows, we derive a contradiction by assuming that $\dim Y=2$.
    In this case, $Y$ is a smooth rationally connected surface by \cite[Theorem 7]{Cutkosky88}, Lemma \ref{lem:small and GFR} and Proposition \ref{thm:GFRtype to RC} (3).

    \textbf{Step 1(The case $\rho(Y)>1$).}\,\,
    First, we consider the case $\rho(Y)>1$.
    In particular, there exists a morphism $g\colon Y\to \PP^1$ whose general fiber is a smooth rational curve.
    Set $h\coloneqq g\circ f\colon X\to \PP^1$, and let $T$ denote a general fiber of $h$.
    Since $X$ is terminal, it has only isolated singularities, and thus $T$ is smooth.
    Set $D\coloneqq -K_X|_T$.
    \begin{cln}\label{cln1}
        $h^0(T,\sO_T(D))=0$.
    \end{cln}
    \begin{proof}
        Suppose by contradiction that $h^0(T,\sO_T(D))>0$.
        Let $(X_A, T_A, D_A)$ be a model of $(X, T, D)$ over a finitely generated $\Z$-subalgebra $A$ of $k$, and fix a general closed point $\mu \in \Spec A$. 
        By Lemma \ref{lem:nef reduction}, $D_{\mu}$ is nef.
    Since $X_{\mu}$ is globally $F$-regular by assumption, the anti-canonical divisor $-K_{X_{\mu}}$ is big by Proposition \ref{rem:GFR}.
    As $T_{\mu}$ is irreducible, it follows from Lemma \ref{lem:restriction} that $D_{\mu}=-K_{X_{\mu}}|_{T_{\mu}}$ is big.
    Now, we obtain
    \[
    0<D_{\mu}^2= D^2=(-K_X)^2\cdot T=0,
    \]
    a contradiction. Thus, our claim follows.
    \end{proof}

    Since $T$ is rationally connected,  
    \[H^2(T,\sO_T(D))\cong H^0(T,\sO_T(K_T-D))=H^0(T,\sO_T(2K_T))=0\] 
    and $\chi(T,\sO_T)=1$.
    By the Riemann-Roch theorem for surfaces, we obtain
    \begin{align*}
        0\overset{\text{Claim}\,\,\ref{cln1}}{=}h^0(T, \sO_T(D))&=h^0(T, \sO_T(D))+h^2(T, \sO_T(D))\\
    &\geq \chi(T,\sO_T(D))\\
    &=D^2+\chi(T,\sO_T)\\
    &= 1,
    \end{align*}
    a contradiction.
   
    \textbf{Step 2(The case $\rho(Y)=1$).}\,\,
    Next, we deal with the case $\rho(Y)=1$.
    Since $Y$ is smooth and rationally connected, we have $Y\cong\PP^2$.
    Let $\ell \subset Y=\PP^2$ be a general line and $T\coloneqq f^{-1}(\ell)$ denote the pullback of $\ell$. 
    Since $X$ has only isolated singularities, $T$ is smooth. 
    We set $D\coloneqq -K_X|_T$.
    
    \begin{cln}\label{cln2}
        $h^0(T,\sO_T(D))=0$.
    \end{cln}
    \begin{proof}
    This follows from essentially the same arguments as in the proof of Claim \ref{cln1}. 
    \end{proof}

    By definition, 
    \[
    D=-K_X|_T=-(K_X+T)|_T+T|_T=-K_T+F,
    \]
    where $F$ is a general fiber of $f|_T:T \to \ell$. 
    Using this and the rationally connectedness of $T$, we have
    \[
    h^2(T,\sO_T(D))=h^0(T,\sO_T(K_T-D))=h^0(T,\sO_T(2K_T-F))=0.
    \]
It then follows from Claim \ref{cln2} and the Riemann-Roch theorem for surfaces that 
    \begin{align*}
        -h^1(T,\sO_T(D))=\chi(X,\sO_T(D))&=\frac{1}{2}D^2+\frac{1}{2}D\cdot(-K_T)+\chi(T, \sO_T)\\
        &=\frac{1}{2}(-K_X)\cdot (-(K_X+T)) \cdot T+1\\
        &=\frac{1}{2}K_X\cdot F+1\\
        &=\frac{1}{2}K_T \cdot F +1\\
        &=0, 
    \end{align*}
    where the fact that $D^2=0$, as a consequence of $\nu(-K_X)=1$, is used for the third equality.
By the exact sequence
\[
H^1(T,\sO_T(D)) \to H^1(T, \sO_T(D+F))\to H^1(F, \sO_F(D+F)),
\]
together with the vanishing 
\[
H^1(F, \sO_F(D+F))=H^1(\sO_F(-K_F))=H^1(\PP^1, \sO_{\PP^1}(2))=0,
\]
we obtain that $H^1(T,\sO_T(D+F))=0$.

Note that $-K_X$ is $f$-ample and $T$ is the pullback of an ample divisor $\ell$, so there exists $0<\epsilon\ll 1$ such that $\epsilon (-K_X)+T$ is ample.
Since $-K_X$ is nef, we know that $-2K_X+T=(2-\epsilon)(-K_X)+\epsilon(-K_X)+T$ is ample, and therefore, by Kodaira vanishing, 
    \[
    H^i(X,\sO_X(-K_X+T))=H^i(X, \sO_X(K_X-2K_X+T))=0
    \]
    for all $i>0$.
    Considering the exact sequence
    \[
    0\to \sO_X(-K_X)\to \sO_X(-K_X+T) \to \sO_T(-K_X+T)=\sO_T(D+F)\to 0,
    \]
    we see that $H^2(X, \sO_X(-K_X))=H^1(T,\sO_T(D+F))=0$.
    
Finally, note that $H^0(X,\sO_X(-K_X))=0$, since $D=-K_X|_T$ is not linearly equivalent to an effective divisor by Claim \ref{cln2}. 
The Riemann-Roch theorem for threefolds \eqref{RR for -K} then implies that  
    \[
    0=h^0(X,\sO_X(-K_X))+h^2(X,\sO_X(-K_X))\geq \chi(X,\sO_X(-K_X))=3,
    \]
which leads to a contradiction.
    \end{proof}

Next, we discuss the case of divisorial contractions. 
Combining similar arguments as in \cite[Theorem 2.1]{Bauer-Peternell} and \cite[Lemma 2.1]{Xie05}, we obtain the following lemma: 

\begin{lem}\label{lem:nu=1,basic facts}
    Let $X$ be a Gorenstein terminal $\Q$-factorial projective threefold over $k$ such that $-K_X$ is nef and $\nu(-K_X)=1$.
    Let $f\colon X\to Y$ be a $K_X$-negative extremal contraction such that 
    $\dim Y=3$. Then the following hold:
    \begin{enumerate}[label=$(\arabic*)$]
        \item $f$ is the blow-up along a locally complete intersection curve $C$ on $Y$ and $Y$ is smooth along $C$.
        \item $K_Y^2\equiv C$.
        \item $-K_Y\cdot C\ge -2$ and $-K_Y\cdot C'\ge0$ for every curve $C'$ satisfying $C'\neq C$. Moreover, if $X$ is smooth, then $-K_Y\cdot C=2(p_a(C)-1)$.
        \item $H^2(Y,\sO_Y(-K_Y))=0$.
    \end{enumerate}
\end{lem}
\begin{proof}
(1)\,
    Since flipping contractions do not occur on Gorenstein terminal $\Q$-factorial threefolds, $f$ contracts a divisor $E$ on $X$. 
Suppose that $f(E)$ is a point on $Y$. 
Then $E$ has Picard number one or $E\cong \mathbb{P}^1_k\times \mathbb{P}^1_k$ by \cite[Theorem 5]{Cutkosky88}. 
In this case, since $-K_X|_E\cdot \ell>0$ for every curve $\ell$ contracted by $f$, it follows that $-K_X|_E$ is an ample divisor on $E$. 
In particular, $(-K_X)^2\cdot E=(-K_X|_E)^2>0$, which contradicts the fact that $\nu(-K_X)=1$. 
Therefore, $f(E)$ is a curve on $Y$, and it then follows from \cite[Theorem 4]{Cutkosky88} that $f\colon X\to Y$ is the blow-up along a locally complete intersection curve $C$ on $Y$ and $Y$ is smooth along $C$.

(2)\, Note that $K_X=f^*K_Y+E$ by (1) and $K_X^2 \equiv 0$ by assumption. Therefore, 
\[
0 \equiv K_X^2=K_X \cdot f^*K_Y+ K_X \cdot E, 
\]
which implies, by pushing forward via $f$, that 
\[
K_Y^2 \equiv f_*(-K_X \cdot E). 
\]

Since $C$ is a locally complete intersection, 
$\mathcal{I}_C/\mathcal{I}^2_C$ is locally free and $E\cong \mathbb{P}(\mathcal{I}_C/\mathcal{I}^2_C)$, where $\mathcal{I}_C$ is the ideal sheaf of $C$ on $Y$. 
Let $F$ be a general fiber of $f|_E:E\to C$.
As in \cite[Lemma 2.1]{Xie05}, we can choose a curve $C_1$ on $E$ such that $\Pic(E)\cong \mathbb{Z}C_1\oplus \mathbb{Z}F$, $f|_{C_1}\colon C_1\to C$ is a birational surjective morphism, and 
\[
-K_X|_E\equiv C_1+bF\,\,\text{and}\,\,E|_E\equiv -C_1+\mu F
\]
for some integers $b$ and $\mu$. 
Thus, $K_Y^2 \equiv f_*(-K_X \cdot E) \equiv f_*C_1 \equiv C$. 

(3)\,
First, take a curve $C'$ on $Y$ such that $C'\neq C$, and let $C'_X$ be the proper transform of $C'$ on $X$.
Then 
\[-K_Y\cdot C'=f^{*}(-K_Y)\cdot C'_X =(-K_X+E)\cdot C'_X\geq 0
\]
since $-K_X$ is nef.

Next, let $e\coloneqq -C_1^2$. Then 
\[2b-e=(-K_X|_E)^2=K_X^2 \cdot E=0.\]
By an argument similar to the proof of \cite[Lemma 2.1]{Xie05}, we have 
\[
-2C_1+(\mu-b)F \equiv (K_X+E)|_E \equiv K_E \equiv -2C_1+(2 p_a(C_1)-2-e)F, 
\]
so that $b+\mu=2(p_a(C_1)-1)$. 
Therefore,
\[
-K_Y\cdot C=-K_X \cdot C_1+E \cdot C_1=b+\mu=2(p_a(C_1)-1)\ge -2.
\]

Finally, if $X$ is smooth, then $C$ must be smooth by \cite[Theorem 3.3]{Mori}, so $C_1\cong C$. 
Thus, $-K_Y\cdot C=2(p_a(C)-1)$.

(4)\,
Note that $Y$ is Gorenstein terminal $\Q$-factorial by (1). 
Take a very ample divisor $H$ on $Y$ such that $-2K_Y+H$ is ample and $C\not\subset H$.
By Bertini theorem, $H$ is smooth and irreducible, and  it follows from Kodaira vanishing that $H^1(Y,\sO_Y(2K_Y-H))=0$.
On the other hand, by (3), we have that $-K_Y\cdot C'\ge 0$ for every curve $C'$ on $Y$ satisfying $C'\neq C$.
Since $C\not\subset H$, we know that $-K_Y|_H$ is nef.
Furthermore, by (2), we have 
\[
(-K_Y|_H)^2=(-K_Y)^2\cdot H=C\cdot H>0, 
\]
so $-K_Y|_H$ is nef and big. 
Applying Kawamata-Viehweg vanishing, we have $H^1(H, \sO_H(2K_Y|_H))=0$, and the exact sequence
\[
H^1(Y,\sO_Y(2K_Y-H)) \to H^1(Y, \sO_Y(2K_Y))\to H^1(H, \sO_H(2K_Y|_H))
\]
implies that $H^2(Y,\sO_Y(-K_Y))\cong H^1(Y, \sO_Y(2K_Y)=0$.  
\end{proof}

\begin{lem}\label{lem:nu=1,birat(singular)}
    Let $X$ be a Gorenstein terminal $\Q$-factorial projective threefold over $k$  of globally $F$-regular type such that $-K_X$ is nef and $\nu(-K_X)=1$.
    Let $f\colon X\to Y$ be a $K_X$-negative extremal contraction such that $\dim Y=3$. Then $-K_Y$ is nef and big.
\end{lem}
\begin{proof}
By Lemma \ref{lem:nu=1,basic facts} (1), $f\colon X\to Y$ is the blow-up along a locally complete intersection curve $C$ on $Y$. 

First, we will show that $-K_Y$ is nef. 
Assume to the contrary that  $-K_Y$ is not nef.
Then, by Lemma \ref{lem:nu=1,basic facts} (3), we have $-2 \le -K_Y\cdot C<0$. 
By the exact sequence
    \[
    0\to f_{*}\sO_X(-K_X)=\mathcal{I}_C\otimes \sO_Y(-K_Y) \to \sO_Y(-K_Y)\to \sO_C(-K_Y)\to 0,
    \]
we obtain $h^0(X,\sO_X(-K_X))=h^0(Y, \sO_Y(-K_Y))$.
By Lemma \ref{lem:nu=1,basic facts} (4), we have $H^2(Y,\sO_Y(-K_Y))=0$. 
Also, $Y$ is of globally $F$-regular type by Lemma \ref{lem:small and GFR} (1)
and in particular, $\chi(Y, \sO_Y)=1$ by Corollary \ref{kodaira vanishing}.
Therefore, by Lemma \ref{lem:nu=1,basic facts} (2), the Riemann-Roch theorem \eqref{RR for -K} yields that
\begin{align*}
    h^0(X,\sO_X(-K_X))=h^0(Y, \sO_Y(-K_Y)) &\ge \chi(Y, \sO_Y(-K_Y))\\
    &=\frac{1}{2}(-K_Y) \cdot C+3\\
    & \ge 2. 
\end{align*}
    Since $\nu(-K_X)=1$, we have the equality $\nu(-K_X)=\kappa(-K_X)$, which implies that $-K_X$ is semi-ample by \cite[Theorem 6.1]{Kawamata(gooddivisors)}.
    It then follows from Corollary \ref{cor:big} that $-K_X$ is big, a contradiction.  Thus, $-K_Y$ must be nef. 

Next, we  show that $-K_Y$ is big. 
Since $K_Y^2\equiv C$ by Lemma \ref{lem:nu=1,basic facts} (2), we have $\nu(-K_Y)\geq 2$.
Applying Proposition \ref{prop:nu=2} to $Y$, we see that $\nu(-K_Y)\neq 2$, and therefore, $-K_Y$ is big.
\end{proof}

In particular, when $X$ is smooth, the following result holds:

\begin{lem}\label{lem:nu=1,birat}
    Let $X$ be a smooth projective threefold over $k$ of globally $F$-regular type such that $-K_X$ is nef.
    Let $f\colon X\to Y$ be a $K_X$-negative extremal  contraction  such that $\dim Y=3$.
    Then $\nu(-K_X)\neq 1$.
\end{lem}
\begin{proof} 
     We use the same notation as in the proof of Lemma \ref{lem:nu=1,birat(singular)}.
    Note that $C$ is smooth because $X$ is smooth (see \cite[Theorem 3.3]{Mori}).  

    Suppose by contradiction that $\nu(-K_X)=1$.
    Since $Y$ is weak Fano by Lemma \ref{lem:nu=1,birat(singular)}, we have $H^i(Y,\sO_Y(-K_Y))=0$ for all $i>0$ by Kawamata-Viehweg vanishing, leading to the isomorphism 
    \[
    H^1(C,\sO_C(-K_Y)) \cong H^2(X,\sO_X(-K_X)). 
    \]
    On the other hand, by Lemma \ref{lem:nu=1,basic facts} (3), 
    \begin{equation}\label{equation for C smooth}      
   \deg K_C=2p_a(C)-2=(-K_Y)\cdot C, 
    \end{equation}
    so that $\deg(K_C+K_Y|_C)=0$. 
    Therefore, 
\[
    h^2(X, \sO_X(-K_X)) = h^1(C,\sO_C(-K_Y))=h^0(C,\sO_C(K_C+K_Y|_C)) \le 1.
\]
It follows from the Riemann-Roch theorem \eqref{RR for -K} that $\chi(X, \sO_X(-K_X))=3$.  
Thus, $h^0(X,\sO_X(-K_X))\geq 2$, and in particular, $\kappa(-K_X)\geq 1$.
Since $\nu(-K_X)=1$ by assumption, we have $\nu(-K_X)=\kappa(-K_X)$, and therefore $-K_X$ is semi-ample by \cite[Theorem 6.1]{Kawamata(gooddivisors)}.
As a result, $-K_X$ must be big by Corollary \ref{cor:big}, leading to a contradiction. 
\end{proof}

\begin{rem}\label{smoothness rem}\,
Lemma \ref{lem:nu=1,birat} is the only statement for which we need the smoothness assumption. If $X$ is singular, then $C$ may also be singular, and \eqref{equation for C smooth} fails in general (see Lemma \ref{lem:nu=1,basic facts} (3)). 
\end{rem}

\begin{proof}[Proof of Proposition \ref{prop:nu=1}]
    By Proposition \ref{prop:nu=0}, we have that $\nu(-K_X)\neq 0$. Thus, we can take a $K_X$-negative extremal  contraction $f\colon X\to Y$ and the assertion follows from Lemmas \ref{lem:nu=1,non-birat} and \ref{lem:nu=1,birat}.
\end{proof}

\subsection{Proof of Theorem \ref{mainthm}}

\begin{proof}[Proof of Theorem \ref{mainthm}]
    It suffices to show that $\nu(-K_X)=3$, but this follows from Propositions \ref{prop:nu=0}, \ref{prop:nu=2}, and \ref{prop:nu=1}.
\end{proof}

\begin{rem}
Ma and Schwede \cite{MaS21} prove a refinement of the correspondence between strongly $F$-regular and klt singularities: a normal $\Q$-Gorenstein singularity $(X,x)$ over a field of characteristic zero is klt if and only if its reduction modulo $p$ is strongly $F$-regular for a single $p$. 
One might expect that Conjecture \ref{Introconj} could be refined in a similar way, but this is not the case. 
\cite[Example 4.16]{Sato-Takagi21}, which is a modification of \cite[Theorem 1.1]{Singh99}, provides an example of complex projective varieties that is not of Fano type, but its reduction modulo $p$ is globally $F$-regular for a single $p$. 
This tells us that the global correspondence, Conjecture \ref{Introconj}, is much more subtle than the local correspondence. 
\end{rem}

\bibliographystyle{alpha}
\bibliography{hoge.bib}

\bigskip

\end{document}